\documentclass[leqno,12pt]{amsart}
\usepackage{amsfonts}
\usepackage{amsmath,amssymb,amsthm}

\everymath{\displaystyle}

\newcommand{\E}{\mathbb E}
\newcommand{\R}{\mathbb R}
\newcommand{\tr}{\mathrm{tr}}
\newcommand{\ds}{\displaystyle}
\newcommand{\manifold}[1]{\mathcal{#1}}
\newcommand{\M}{\manifold{M}}

\renewcommand{\span}{\mathrm{span}}
\newcommand{\const}{\mathrm{const}}

\newtheorem{thm}{Theorem}[section]
\newtheorem{cor}[thm]{Corollary}

\newtheorem{prop}[thm]{Proposition}
\theoremstyle{definition}

\theoremstyle{remark}
\newtheorem{rem}[thm]{Remark}

\numberwithin{equation}{section}

\begin{document}

\title[Timelike Meridian Surfaces  of hyperbolic type in the Minkowski 4-Space]{Timelike Meridian Surfaces of hyperbolic type in the Minkowski 4-Space}

\author{Victoria Bencheva, Velichka Milousheva}

\address{''St.  Cyril and St. Methodius'' University of Veliko Tarnovo,  Faculty of Mathematics and Informatics, Arhitekt Georgi Kozarov Str. 1, 5000 Veliko Tarnovo, Bulgaria}
\address{Institute of Mathematics and Informatics, Bulgarian Academy of Sciences, Acad. G. Bonchev Str. bl. 8, 1113, Sofia, Bulgaria}

\email{viktoriq.bencheva@gmail.com}
\email{vmil@math.bas.bg}


\begin{abstract}

We consider  a special class of timelike surfaces in the
four-dimensional Min\-kow\-ski space which are one-parameter systems of meridians of  rotational hypersurfaces with spacelike axis and call them meridian surfaces of hyperbolic type. We show that all timelike meridian surfaces of hyperbolic type are surfaces with flat normal connection and give the complete classification of those surfaces with constant Gauss curvature. 
We also classify all minimal timelike meridian surfaces of hyperbolic type and all timelike  meridian surfaces of hyperbolic type  with non-zero constant mean curvature (CMC-surfaces). 
We show that there are no timelike meridian surfaces of hyperbolic type with parallel mean curvature vector field other than CMC surfaces lying in a hyperplane. 
Finally, we describe the class of  timelike meridian surfaces of hyperbolic type  with parallel normalized mean curvature vector field, but non-parallel mean curvature vector.

\end{abstract}

\maketitle

\section{Introduction}

One of the basic problems in the contemporary
differential geometry of surfaces  and hypersurfaces in standard model spaces such as the Euclidean space $\R^n$ and the pseudo-Euclidean space $\R^n_k$ (with arbitrary dimension $n$ and arbitrary index $k$) is the  study of the main invariants characterizing the surfaces. 
Curvature invariants are the most fundamental and natural invariants in Riemannian and pseudo-Riemannian geometry and also play key roles in physics. 
The basic intrinsic curvature invariant of a surface in   Euclidean or pseudo-Euclidean space is the
Gauss curvature and one basic extrinsic invariant is the curvature of the normal
connection.  So, a common problem is to investigate various important classes of surfaces  which are characterized by conditions on their main curvature invariants. For example, flat surfaces, surfaces with constant Gauss curvature, surfaces with flat normal connection, etc. 

The mean curvature vector field is the most fundamental normal vector field on an immersed surface, capturing its extrinsic shape.
So, a fundamental   question is also to investigate different classes of surfaces characterized by conditions on the mean curvature vector field, such as surfaces with constant mean curvature (CMC surfaces), surfaces with parallel mean curvature vector field or parallel normalized mean curvature vector field,  and to find examples of surfaces belonging to these classes. 

On the other hand, a basic source of examples of many geometric classes of surfaces in Riemannian and pseudo-Riemannian geometry  are the rotational surfaces and hypersurfaces.
In the present paper, we consider a special class of timelike surfaces in the 4-dimensional Minkowski space $\R^4_1$ which are one-parameter systems of meridians of  rotational hypersurfaces with spacelike axis and call them meridian surfaces of hyperbolic type. They are analogous to the timelike meridian surfaces of elliptic type, which are constructed as one-parameter systems of meridians of  rotational hypersurfaces with timelike axis and are studied in \cite{BM-JGeom}.  Depending on the type of the spheres in the 3-dimensional subspace $\R^3_1 \subset \R^4_1$ 
and the casual character of the spherical curves, we distinguish two different kinds of timelike meridian surfaces of hyperbolic type in  $\R^4_1$.

We show that all timelike meridian surfaces of hyperbolic type are surfaces with flat normal connection and  classify completely  the flat meridian surfaces of hyperbolic type 
(Theorem \ref{Th:meridSurfFlat} and Theorem \ref{Th:meridSurfFlat2}) and the meridian surfaces of hyperbolic type  with
constant non-zero Gauss curvature (Theorem \ref{Th:meridSurfConstK} and Theorem \ref{Th:meridSurfConstK2}). 

In Theorem \ref{Th:meridSurfMin-1}  and Theorem \ref{Th:meridSurfMin-2} we classify all minimal timelike first or second kind meridian surfaces of hyperbolic type.
Then we describe all timelike first or second kind meridian surfaces of hyperbolic type  with non-zero constant mean curvature (CMC-surfaces), see Theorem \ref{Th:meridSurfConstH} and Theorem \ref{Th:meridSurfConstH2}. We also consider timelike first or second kind meridian surfaces of hyperbolic type with  parallel mean curvature vector field and prove that there are no timelike meridian surfaces of hyperbolic type with parallel mean curvature vector field other than CMC-surfaces lying in a hyperplane of $\R^4_1$.

Finally, we study the class of  timelike meridian surfaces of hyperbolic type  with parallel normalized mean curvature vector field and classify them in Theorem \ref{Th:meridSurfParallelNorm} and Theorem \ref{Th:meridSurfParallelNorm-2}.  The results in these two theorems can be used to find explicit solutions to the background systems of natural partial differential equations describing the timelike surfaces with parallel normalized mean curvature vector field in $\R^4_1$, see \cite{BM-TJM}.

\section{Preliminaries}

We consider the four-dimensional Minkowski space $\R^4_1$  endowed with the standard flat metric
of signature $(3,1)$ given in local coordinates by
$$\widetilde{g} = dx_1^2 + dx_2^2 + dx_3^2 -dx_4^2,$$
where $(x_1; x_2; x_3; x_4)$ is a rectangular coordinate system in $\R^4_1$. 
Denote by $\langle , \rangle$ the indefinite inner scalar product associated with $\widetilde{g}$.
Since $\widetilde{g}$ is an
indefinite metric, a vector $v \in \R^4_1$ can have one of the three
casual characters:  \textit{spacelike} if $\langle v, v \rangle >0$
or $v=0$, \textit{timelike} if $\langle v, v \rangle<0$, and \textit{null} (%
\textit{lightlike}) if $\langle v, v \rangle =0$ and $v\neq 0$.

We use the following standard denotations:
\begin{equation*}
\begin{array}{l}
\vspace{2mm}
\mathbb{S}^3_1(1) =\left\{V\in \R^4_1: \langle V, V \rangle =1 \right\}; \\
\vspace{2mm}
\mathbb{H}^3_1(-1) =\left\{ V\in \R^4_1: \langle V, V \rangle = -1\right\}.
\end{array}
\end{equation*}
The space $\mathbb{S}^3_1(1)$ is known as the de Sitter space, and the
space $\mathbb{H}^3_1(-1)$ is known as  the hyperbolic space (or the anti-de Sitter space) \cite{O'N}.

Given a surface $\mathcal{M}^2$ in $\R^4_1$, we denote by $g$ the induced metric
of $\widetilde{g}$ on $\mathcal{M}^2$. A surface $\mathcal{M}^2$ in $\R^4_1$ is called \emph{timelike} (or \emph{Lorentz}) if the induced metric $g$ on $\mathcal{M}^2$ is Lorentzian. 
In the present paper
we study timelike surfaces, so at each point
$p\in \mathcal{M}^2$ we have the following decomposition 
\begin{equation*}
\R^4_1 = T_p\mathcal{M}^2 \oplus N_p\mathcal{M}^2
\end{equation*}
with the property that the restriction of the metric onto the tangent space $T_p\mathcal{M}^2$ is of signature $(1,1)$, and the restriction of the metric onto the
normal space $N_p\mathcal{M}^2$ is of signature $(2,0)$.

Let $\widetilde{\nabla}$ and $\nabla$ be the Levi Civita connections on $\R^4_1$ and $\mathcal{M}^2$, respectively. Then, for any tangent vector fields  $x$, $y$ and any normal vector field $\xi$ we have the 
following formulas of Gauss and Weingarten:
$$\begin{array}{l}
\vspace{2mm}
\widetilde{\nabla}_xy = \nabla_xy + \sigma(x,y);\\
\vspace{2mm}
\widetilde{\nabla}_x \xi = - A_{\xi} x + D_x \xi,
\end{array}$$
which define the second fundamental tensor $\sigma$, the normal connection $D$
and the shape operator $A_{\xi}$ with respect to $\xi$.

The mean curvature vector  field $H$ of $\mathcal{M}^2$ is defined as
$H = \ds{\frac{1}{2}\,  \tr\, \sigma}$, i.e. in the case of a timelike surface, 
$H = \ds{\frac{1}{2} \left( \sigma(x,x) - \sigma(y,y)\right)}$, where   $\{x,y\}$
is a  local orthonormal frame  of the tangent bundle, such that $\langle x, x \rangle = 1$, $\langle y, y \rangle = - 1$.

A surface $\M^2$ is called \textit{minimal} if its
mean curvature vector vanishes identically, i.e.  $H=0$. If $\langle H, H \rangle = const \neq 0$, then the surface 
 is said to have constant mean curvature (or $\M^2$  is a \textit{CMC-surface}).
A normal vector field $\xi$ on  $\M^2$ is called \emph{parallel in the normal bundle} (or simply \emph{parallel}) if $D{\xi}=0$.
If the mean curvature vector $H$ is parallel, i.e.
$D H =0$, then the surface  is said to have \emph{parallel mean curvature vector field} \cite{Chen}.  Surfaces with parallel mean curvature vector field in Riemannian
space forms were classified by Chen \cite{Chen1}  and Yau \cite{Yau} in the early 1970s.
Lorentz surfaces with parallel mean curvature vector field in arbitrary pseudo-Euclidean space $\R^n_k$
were studied in \cite{Chen-KJM}  and \cite{Fu-Hou}.

Surfaces for which the mean curvature vector field $H$ is non-zero, i.e. $\langle
H, H \rangle \neq 0$, and  the unit vector field in the
direction of $H$  is parallel in the normal bundle, are called surfaces with \textit{parallel normalized mean
curvature vector field} \cite{Chen1}. It can easily be seen that if $\M^2$ is a
surface with non-zero parallel mean curvature vector field $H$ (i.e. $H \neq 0$ and $DH = 0$%
), then $\M^2$ is a surface with parallel normalized mean curvature vector
field, but the converse is true only in the case $%
\Vert H \Vert = \const$ \cite{Chen-MM}.

\section{Construction of timelike meridian surfaces of hyperbolic type in  $\R^4_1$}

Meridian surfaces in the Euclidean 4-space $\R^4$ we defined first by G. Ganchev and the  second author in \cite{GM2}  as 2-dimensional surfaces lying on a  rotational hypersurface in $\R^4$. As these surfaces are one-parameter systems of meridians of
the rotational hypersurface, they are called \textit{meridian surfaces}.
The classification of  meridian surfaces with constant Gauss curvature, with constant mean curvature, Chen meridian surfaces and  meridian surfaces with parallel normal bundle was given in \cite{GM2}  and \cite{GM-BKMS}. 

 In \cite{GM6},  the idea from the Euclidean case was used to construct special families of spacelike surfaces lying on rotational hypersurfaces with timelike or spacelike axis in the Minkowski space $\R^4_1$ and these surfaces were called spacelike meridian surface of elliptic  or hyperbolic type, respectively. A local classification of marginally trapped meridian surfaces in the  Minkowski 4-space was given  in \cite{GM6} and meridian surfaces with pointwise 1-type Gauss map in $\R^4_1$  were classified in \cite{AM}. 

In \cite{BM-JGeom}, the authors studied timelike meridian surfaces of elliptic type in the Minkowski 4-space -- these are 2-dimensional timelike surfaces
which lie on rotational hypersurfaces with timelike axis. The following basic classes of timelike meridian surfaces of elliptic type were studied:
with constant Gauss curvature, with constant mean curvature, with parallel mean curvature vector field, with parallel normalized mean curvature
vector field.

 In the present paper,  we develop the same idea and consider timelike surfaces in  $\R^4_1$ which are one-parameter systems of meridians of rotational hypersurfaces with  spacelike axis and call them meridian surfaces of hyperbolic type. Depending on the type of the spheres in the 3-dimensional subspace $\R^3_1 \subset \R^4_1$ 
and the casual character of the spherical curves, we distinguish different kinds of timelike meridian surfaces of hyperbolic type in  $\R^4_1$.

Bellow we give the construction of these different kinds of meridian surfaces.

Let $Oe_1 e_2 e_3 e_4$ be the standard orthonormal frame  in $\R^4_1$, i.e. $\langle e_1,e_1 \rangle =\langle e_2,e_2 \rangle = \langle e_3,e_3 \rangle = 1$, $\langle e_4,e_4 \rangle = -1$. We shall
consider a rotational hypersurface with spacelike axis $Oe_1$. Of course, similarly  one
can consider a rotational hypersurface with axis $Oe_2$ or $Oe_3$.
Since in the three-dimensional  Minkowski subspace $\E^3_1 = \span\left \{ e_{2},e_{3},e_{4}\right \}$ there exist two types of  spheres, namely the  pseudo-sphere  $\mathbb{S}^2_1(1) =\left\{V\in \E^3_1: \langle V, V \rangle = 1\right \}$ and the pseudo-hyperbolic space $\mathbb{H}^2(-1)=\left\{V \in  \E^3_1: \langle V, V \rangle = - 1; V^3 >0\right \}$ (where $V^3$ is the third coordinate of $V$), we can consider two kinds of rotational hypersurfaces about the axis  $Oe_1$.

Let $f = f(u), \,\, g = g(u)$, $u \in I$ be smooth functions, defined in an
interval $I \subset \R$, such that $\dot{f}^2(u) + \dot{g}^2(u) \neq 0$, $f(u)>0, \,\, u \in I$.

The first kind rotational hypersurface $\M^I$ in $\R^4_1$ obtained by the rotation of the
meridian curve $m:u\rightarrow (f(u),g(u))$ about the $Oe_1$-axis, is
parametrized as follows: 
$$\M^I: Z(u,w^1,w^2) = g(u)e_1 +  f(u) \left( \cosh w^1 \cos w^2 e_2 +  \cosh w^1 \sin w^2 e_3+ \sinh w^1 e_4\right).$$
Note that
 $$l^{I}(w^1,w^2) = \cosh w^1 \cos w^2 \,e_2 +  \cosh w^1 \sin w^2 \,e_3 + \sinh w^1 \,e_4$$ 
is the unit position vector of the pseudo-sphere $\mathbb{S}^2_1(1)$ (de Sitter space)
 in $\E^3_1$ centered at the origin $O$. Then, the parametrization of the rotational hypersurface $\M^I$ can be written as:
$$\M^I: Z(u,w^1,w^2) = g(u)e_1 +  f(u) l^{I}(w^1,w^2).$$

Now, we will construct surfaces in $\E^4_1$ which are one-parameter systems of meridians of the hypersurface $\M^I$.
Let $w^1 = w^1(v), \, w^2=w^2(v), \,\, v \in J, \,J \subset
\R$ and consider a  smooth timelike curve $\textbf{c}:l=l(v)=l^I(w^1(v),w^2(v)), \,\, v \in J, \, J \subset \R$
on $\mathbb{S}^2_1(1)$ parametrized by the arc-length, i.e.  $\langle l\,',l\,'\rangle = -1$. Now, we consider the two-dimensional surface $\M^I_m$ in $\R^4_1$ parametrized by:
\begin{equation}  \notag
\M^I_m:  z(u,v) = Z(u,w^1(v),w^2(v)),\quad u \in I, \, v \in J.
\end{equation}
 The parametrization of $\M^I_m$ can also be written  in the following form:
\begin{equation} \label{E:Eq3}
\M^I_m: z(u,v) = f(u) \, l(v) + g(u)\, e_1, \quad u \in I, \, v \in J.
\end{equation}

Obviously,  $\M^I_m$ is a one-parameter system of
meridians of the hypersurface $\M^I$, and for this reason we call $\M^I_m$ a
\emph{first kind meridian surface of hyperbolic type}. 

\vskip 2mm
The second kind rotational hypersurface $\M^{II}$ in $\R^4_1$ obtained by the rotation of the
meridian curve $m:u\rightarrow (f(u),g(u))$ about the $Oe_1$-axis, is given by the following parametrization:
$$\M^{II}: Z(u,w^1,w^2) = g(u)e_1 +  f(u) \left( \sinh w^1 \cos w^2 e_2 +  \sinh w^1 \sin w^2 e_3+ \cosh w^1 e_4\right).$$
Note that  $$l^{II}(w^1,w^2) = \sinh w^1 \cos w^2 \,e_2 +  \sinh w^1 \sin w^2 \,e_3 + \cosh w^1 \,e_4$$  is the unit position vector of the hyperbolic sphere $\mathbb{H}^2(-1)$ (anti de Sitter space)
 in $\E^3_1 = \span\left \{ e_{2},e_{3},e_{4}\right \}$ centered at the origin $O$. So,  the parametrization of
$\M^{II}$ can be written as:
\begin{equation*}
\M^{II}: Z(u,w^1,w^2) = g(u) \,e_1 + f(u) l^{II}(w^1,w^2).
\end{equation*}

Meridian surfaces lying on the second kind rotational hypersurface  $\M^{II}$ can be constructed as follows. Let $\textbf{c}: l = l(v) =  l^{II}(w^1(v),w^2(v))$ be a smooth curve on the hyperbolic sphere $\mathbb{H}^2(-1)$, where $w^1 = w^1(v)$, $w^2=w^2(v), \,\, v \in J, \, J \subset \R$.  Then, the  two-dimensional surface $\M^{II}_m$  defined by
\begin{equation}  \label{E:Eq-4}
\M^{II}_m: z(u,v) = f(u)\, l(v) + g(u) \,e_1, \quad u \in I, \, v \in J
\end{equation}
 is a  one-parameter system of meridians of the second kind rotational hypersurface $\M^{II}$, which we call a  \emph{second kind meridian surface of hyperbolic type}.

Spacelike meridian surfaces of hyperbolic type in the Minkowski 4-space were studied in \cite{GM-MC} where  the classification of  such surface with constant Gauss curvature or with constant mean curvature was given.

In this paper we will study the timelike case of meridian surfaces of hyperbolic type.

\section{Main invariants of timelike meridian surfaces of hyperbolic type}

In this section we will study some main invariants of the two kinds of timelike meridian surfaces of hyperbolic type.
Note that the first kind meridian surface of hyperbolic type $\M^I_m$ is generated by a smooth timelike curve $\textbf{c}$ lying on the pseudo-sphere $\mathbb{S}^2_1(1)$, while the second kind meridian surface of hyperbolic type $\M^{II}_m$ is generated by a smooth timelike curve $\textbf{c}$ lying on the hyperbolic sphere $\mathbb{H}^2(-1)$. 

 For convenience, we use the notation $\dot{f}(u) =  \frac{\partial f}{\partial u}$ for differentiation with respect to the parameter $u$,  and we use $l\,'(v) =  \frac{\partial l}{\partial v}$ for differentiation with respect to the parameter $v$.

\subsection{First kind meridian surfaces of hyperbolic type}

Let $\M^I_m$ be the first kind meridian surface of hyperbolic type,  defined by \eqref{E:Eq3}. The Frenet formulas for the timelike curve $\textbf{c}$ on $\mathbb{S}^2_1(1)$ are given bellow:
$$
\begin{array}{l}
\vspace{2mm}
l\,' = t; \\
\vspace{2mm}
t\,'=  \varkappa \,n + l; \\
\vspace{2mm}
n\,'= \varkappa \,t,
\end{array}
$$
where $\varkappa = \varkappa(v)$ is the spherical curvature of $\textbf{c}$, i.e. $\varkappa (v)= \langle t'(v), n(v) \rangle$, and $\{l(v), t(v), n(v)\}$ is an orthonormal frame field in $\R^3 = \span \{e_2, e_3, e_4\}$, such that $\langle l, l \rangle = 1$, $\langle t, t \rangle = -1$, $\langle n, n \rangle = 1$.

Then, the tangent vector fields $z_u$ and $z_v$ of $\M^I_m$  are expressed as follows:
$$
\begin{array}{l}
\vspace{2mm}
z_u = \dot{f}(u) \,l(v) + \dot{g}(u) \,e_1;\\
\vspace{2mm}
z_v =  f(u) \,t(v),
\end{array}
$$
so, the coefficients of the first fundamental form of $\M^I_m$ are
$$E = \langle z_u, z_u \rangle = \dot{f}^2(u) +\dot{g}^2(u); \quad F = \langle z_u, z_v \rangle = 0; \quad 
G = \langle z_v, z_v \rangle = -f^2(u).$$
Obviously, $\M^I_m$  is a timelike surface ($E>0$ and $G<0$).

%??Since we are interested in timelike surfaces, we consider the case $\dot{f}^2(u) - \dot{g}^2(u) <0, \,\, u \in I$. Without loss of generality we may assume that $\dot{f}(u)^2-\dot{g}(u)^2 = -1$. Hence, it follows that  $\dot{f}\ddot{f}-\dot{g}\ddot{g} = 0$. ??

The curvature $\varkappa_m$ of the meridian curve $m$ is given by $\varkappa_m(u) = \dot{f}(u) \ddot{g}(u) - \dot{g}(u) \ddot{f}(u)$.

Without loss of generality we may assume that $\dot{f}(u)^2+\dot{g}(u)^2 = 1$. Then, the coefficients of first fundamental form are:
$$E = 1;\quad F = 0;\quad G =  -f^2(u),$$
and hence,  $EG-F^2 = -f^2(u)$. 
We consider the following orthonormal tangent frame field:
\begin{equation} \label{E:Eq-mer-tangent}
\begin{array}{l}
\vspace{2mm}
X = z_u;  \\
\vspace{2mm}
Y = \ds \frac{z_v}{f} = t.
\end{array}
\end{equation}
Therefore, $\langle X,X \rangle = 1, \,\,  \langle X,Y \rangle = 0, \,\,  \langle Y,Y \rangle = -1$. 
We consider the following orthonormal frame field of the  normal space:
\begin{equation} \label{E:Eq-mer-normal}
\begin{array}{l}
\vspace{2mm}
N_1 = n(v); \\
\vspace{2mm}
N_2 =  \dot{g}(u) \,l(v) - \dot{f}(u)\, e_1.
\end{array}
\end{equation}
Obviously, 
$\langle N_1,N_1 \rangle = 1, \,\,  \langle N_1,N_2 \rangle = 0, \,\,  \langle N_2,N_2 \rangle = 1$. 
Calculating the second derivatives of the vector function $z(u,v)$ we obtain:
\begin{equation} \label{E:Eq-mer-second-deriv}
\begin{array}{l}
\vspace{2mm}
z_{uu} = \ddot{f}\, l + \ddot{g}\, e_1; \\
\vspace{2mm}
z_{uv} = \dot{f}\, l' = \dot{f}\, t; \\
\vspace{2mm}
z_{vv} = f \, t' = f \varkappa \,n + f \,l, 
\end{array}
\end{equation}
which imply that 
$$
\begin{array}{ll}
\vspace{2mm}
\langle z_{uu}, N_1 \rangle = 0; & \qquad  \langle z_{uu}, N_2 \rangle = -\varkappa_m;\\
\vspace{2mm}
 \langle z_{uv}, N_1 \rangle = 0; & \qquad  \langle z_{uv}, N_2 \rangle = 0;\\
\vspace{2mm}
\langle z_{vv}, N_1 \rangle = f \varkappa; & \qquad  \langle z_{vv}, N_2 \rangle =  f \dot{g}.
\end{array}
$$

Using that  $X = z_u$, $Y = \ds \frac{z_v}{f}$ and formulas \eqref{E:Eq-mer-second-deriv}, we  obtain the following  derivative formulas for the orthonormal frame field $\{X, Y, N_1, N_2\}$:
$$
\begin{array}{ll}
\vspace{2mm}
\widetilde{\nabla}_X X = \ddot{f} \,l + \ddot{g}\, e_1; \qquad \quad & \widetilde{\nabla}_X N_1 = 0;\\
\vspace{2mm}
\widetilde{\nabla}_X Y = 0;  \qquad \quad & \widetilde{\nabla}_Y N_1 = \ds \frac{\varkappa}{f}\, t;\\
\vspace{2mm}
\widetilde{\nabla}_Y X = \ds \frac{\dot{f}}{f}\, t;  \qquad \quad& \widetilde{\nabla}_X N_2 = \ddot{g} \,l - \ddot{f} \,e_1; \\
\vspace{2mm}
\widetilde{\nabla}_Y Y = \ds \frac{\varkappa}{f}\, n + \ds \frac{1}{f}\, l;  \qquad \quad & \widetilde{\nabla}_Y N_2 = \ds \frac{\dot{g}}{f}\, t.
\end{array} 
$$
Now, having in mind  \eqref{E:Eq-mer-tangent} and \eqref{E:Eq-mer-normal}, the formulas above imply:
\begin{equation} \label{E:Eq-mer-derivative-1}
\begin{array}{ll}
\vspace{2mm} 
\widetilde{\nabla}_X X = \quad \quad \qquad \qquad - \varkappa_m \,N_2; \qquad \quad & \widetilde{\nabla}_X \,N_1 = 0;\\
\vspace{2mm}
\widetilde{\nabla}_X Y = 0; \qquad \quad & \widetilde{\nabla}_Y\, N_1 = \quad \quad \ds \frac{\varkappa}{f} \,Y;\\
\vspace{2mm}
\widetilde{\nabla}_Y X = \qquad \ds \frac{\dot{f}}{f}\,Y; \qquad \quad & \widetilde{\nabla}_X \,N_2 =  \varkappa_m \,X; \\
\vspace{2mm}
\widetilde{\nabla}_Y Y = \ds \frac{\dot{f}}{f} \,X  \quad +\frac{\varkappa}{f} \,N_1 + \frac{\dot{g}}{f}\, N_2; \qquad \quad & \widetilde{\nabla}_Y \,N_2 =  \quad \qquad \ds \frac{\dot{g}}{f} \,Y. 
\end{array} 
\end{equation}
Equations \eqref{E:Eq-mer-derivative-1} are the derivative formulas of the surface with respect to the parameters $(u,v)$.
Using these formulas we will calculate the Gauss curvature $K$, the curvature of the normal connection $K^{\bot}$, and the mean curvature vector field $H$  of the first kind meridian surface of hyperbolic type $\M^I_m$.

\vskip 1mm
The Gauss curvature $K$ is determined by the following formula:
$$
K = \frac{\langle \sigma(X,X) , \sigma(Y,Y)\rangle - \langle \sigma(X,Y) , \sigma(X,Y)\rangle}{\langle X,X\rangle \langle Y,Y\rangle - \langle X,Y\rangle ^2}.
$$
Hence, using \eqref{E:Eq-mer-derivative-1}, we get that the  Gauss curvature of $\M^I_m$ is expressed as follows: 
\begin{equation}\label{Eq:meridianSurfGaussCurv}
K = -\frac{\ddot{f}(u)}{f(u)}.
\end{equation}
The above formula shows that the Gauss curvature depends only on the parameter $u$ since it is expressed in terms of the function $f(u)$ determining the meridian curve $m$.

Using \eqref{E:Eq-mer-derivative-1} we obtain that the normal mean curvature vector field $H$  depends on both parameters $u$ and $v$ and is given by the formula:
\begin{equation} \label{E:Eq-mer-H}
H = -\frac{\varkappa (v)}{2f(u)} \, N_1 \pm \frac{f(u) \ddot{f}(u) + \dot{f}^2(u) -1}{2f(u) \sqrt{1-\dot{f}^2(u)}} \, N_2.
\end{equation}

The curvature of the normal connection of the surface is defined by:
$$
K^{\bot} = \frac{\langle R^D(X,Y)N_1, N_2\rangle}{\langle X,X\rangle \langle Y,Y\rangle - \langle X,Y\rangle ^2},
$$
where 
$$R^D(X,Y)N = D_X D_Y N - D_Y D_X N - D_{[X,Y]}N.$$
Using equalities  \eqref{E:Eq-mer-derivative-1}, we calculate that the  curvature of the normal connection of  $\M^I_m$ is:
$$K^{\bot}=0.$$
The above formula implies the next statement.

\begin{prop}
The first kind meridian surface of hyperbolic type $\M^I_m$, defined by \eqref{E:Eq3}, is a surface with flat normal connection.
\end{prop}

\vskip 2mm
\begin{rem}
Each timelike first kind meridian surface of  hyperbolic type is determined by a meridian curve $m$ and a  timelike curve $\textbf{c}$ lying on the pseudo-sphere $\mathbb{S}^2_1(1)$ in $\R^3_1$. So, the geometry of the first kind meridian surface and all its invariants are expressed in terms of the curvature $\varkappa_m(u)$ of the meridian curve $m$ and the spherical curvature $\varkappa (v)$ of  $\textbf{c}$.
\end{rem}

\begin{rem}

(i) If the spherical curvature of $\textbf{c}$ is $\varkappa =0$, then from \eqref{E:Eq-mer-derivative-1} it follows that $\widetilde{\nabla}_X N_1 = \widetilde{\nabla}_X N_2= 0$ and hence, the normal vector field $N_1$ is constant. So, the meridian surface $\M^I_m$ lies in the
constant 3-dimensional space  $\R_1^3 = \span \{X,Y,N_2\}$ of $\R_1^4$.

\vskip 1mm
(ii) If $\varkappa_m = 0$, i.e. the meridian curve $m$ is part of a straight  line, then both the Gauss curvature and the normal curvature are zero, which implies that $\M^I_m$ is a flat surface with flat normal connection. Since in this case the surface consists only of inflection points, then it is either developable  or  lies in a 3-dimensional space \cite{Lane}.
\end{rem}

\subsection{Second kind meridian surfaces of hyperbolic type}

Let $\M^{II}_m$ be the second kind meridian surface of hyperbolic type,  defined by  \eqref{E:Eq-4}. 
The Frenet formulas for a curve $\textbf{c}$ on $\mathbb{H}^2_1(-1)$ are given bellow:
$$
\begin{array}{l}
\vspace{2mm}
l\,' = t; \\
\vspace{2mm}
t\,'=  \varkappa \,n + l; \\
\vspace{2mm}
n\,'= - \varkappa \,t,
\end{array}
$$
where $\varkappa = \varkappa(v)$ is the spherical curvature of $\textbf{c}$, i.e. $\varkappa (v)= \langle t\,'(v), n(v) \rangle$, and $\{l(v), t(v), n(v)\}$ is an orthonormal frame field in $\R^3 = \span \{e_2, e_3, e_4\}$, such that $\langle l, l \rangle = -1$, $\langle t, t \rangle = 1$, $\langle n, n \rangle = 1$.

Then, the tangent vector fields $z_u$ and $z_v$ of $\M^{II}_m$ are expressed as:
$$
\begin{array}{l}
\vspace{2mm}
z_u = \dot{f}(u) \,l(v) + \dot{g}(u) \,e_1;\\
\vspace{2mm}
z_v =  f(u) \,t(v).
\end{array}
$$
Then, the coefficients of the first fundamental form of $\M^{II}_m$ are:
$$E = \langle z_u, z_u \rangle = - \dot{f}^2(u) +\dot{g}^2(u); \quad F = \langle z_u, z_v \rangle = 0; \quad 
G = \langle z_v, z_v \rangle = f^2(u).$$
Since we are interested in timelike surfaces, we assume that $-\dot{f}^2(u) + \dot{g}^2(u) <0, \,\, u \in I$.
Without loss of generality we may assume that $\dot{f}(u)^2-\dot{g}(u)^2 = 1$. Then, the coefficients of the first fundamental form are:
$$E = - 1; \quad F =  0; \quad G = f^2(u).$$
Again we have  $EG-F^2 = -f^2(u)$. 
Let us consider the following orthonormal tangent frame field:
\begin{equation} \label{E:Eq-mer-tangent2}
\begin{array}{l}
\vspace{2mm}
X = z_u;  \\
\vspace{2mm}
Y = \ds \frac{z_v}{f} = t.
\end{array}
\end{equation}
Therefore, $\langle X,X \rangle = -1, \,\,  \langle X,Y \rangle = 0, \,\,  \langle Y,Y \rangle = 1$. 
We consider the following orthonormal frame field of the  normal space:
\begin{equation} \label{E:Eq-mer-normal2}
\begin{array}{l}
\vspace{2mm}
N_1 = n(v); \\
\vspace{2mm}
N_2 =  \dot{g}(u) \,l(v) + \dot{f}(u) \,e_1.
\end{array}
\end{equation}
Obviously, 
$\langle N_1,N_1 \rangle = 1, \,\,  \langle N_1,N_2 \rangle = 0, \,\,  \langle N_2,N_2 \rangle = 1$. 
Calculating the second derivatives of the vector function $z(u,v)$ we obtain:
\begin{equation} \label{E:Eq-mer-second-deriv2}
\begin{array}{l}
\vspace{2mm}
z_{uu} = \ddot{f}\, l + \ddot{g} \,e_1; \\
\vspace{2mm}
z_{uv} = \dot{f}\,l' = \dot{f} \,t; \\
\vspace{2mm}
z_{vv} = f \,t' = f \varkappa \,n + f \,l, 
\end{array}
\end{equation}
which imply that 
$$
\begin{array}{ll}
\vspace{2mm}
\langle z_{uu}, N_1 \rangle = 0; & \qquad  \quad \langle z_{uu}, N_2 \rangle = \varkappa_m;\\
\vspace{2mm}
 \langle z_{uv}, N_1 \rangle = 0; & \qquad \quad \langle z_{uv}, N_2 \rangle = 0;\\
\vspace{2mm}
\langle z_{vv}, N_1 \rangle = f \varkappa; & \qquad \quad \langle z_{vv}, N_2 \rangle =  -f \dot{g}.
\end{array}
$$

Using that  $X = z_u$, $Y = \ds \frac{z_v}{f}$ and formulas \eqref{E:Eq-mer-second-deriv2}, we  obtain the following  derivative formulas for the frame field $\{X, Y, N_1, N_2\}$:
$$
\begin{array}{l}
\vspace{2mm}
\widetilde{\nabla}_X X = \ddot{f} l + \ddot{g} e_1;\\
\vspace{2mm}
\widetilde{\nabla}_X Y = 0;\\
\vspace{2mm}
\widetilde{\nabla}_Y X = \ds \frac{\dot{f}}{f} t;\\
\vspace{2mm}
\widetilde{\nabla}_Y Y = \ds \frac{\varkappa}{f} n + \ds \frac{1}{f} l;
\end{array} \qquad \quad
\begin{array}{l}
\vspace{2mm}
\widetilde{\nabla}_X N_1 = 0;\\
\vspace{2mm}
\widetilde{\nabla}_Y N_1 = -\ds \frac{\varkappa}{f} t;\\
\vspace{2mm}
\widetilde{\nabla}_X N_2 = \ddot{g} l + \ddot{f} e_1; \\
\vspace{2mm}
\widetilde{\nabla}_Y N_2 = \ds \frac{\dot{g}}{f} t.
\end{array} 
$$
Now, having in mind  \eqref{E:Eq-mer-tangent2} and \eqref{E:Eq-mer-normal2}, the formulas above imply:
\begin{equation} \label{E:Eq-mer-derivative-2}
\begin{array}{ll}
\vspace{2mm}
\widetilde{\nabla}_X X = \quad \quad \qquad \qquad  \varkappa_m N_2; & \qquad \quad \widetilde{\nabla}_X N_1 = 0;\\
\vspace{2mm}
\widetilde{\nabla}_X Y = 0; & \qquad \quad \widetilde{\nabla}_Y N_1 = \quad \quad \ds -\frac{\varkappa}{f} Y;\\
\vspace{2mm}
\widetilde{\nabla}_Y X = \qquad \ds \frac{\dot{f}}{f}Y; & \qquad \quad \widetilde{\nabla}_X N_2 =  -\varkappa_m X; \\
\vspace{2mm}
\widetilde{\nabla}_Y Y = \ds -\frac{\dot{f}}{f} X  \quad +\frac{\varkappa}{f} N_1 - \frac{\dot{g}}{f} N_2; & \qquad \quad \widetilde{\nabla}_Y N_2 =  \quad \qquad \ds \frac{\dot{g}}{f} Y. 
\end{array} 
\end{equation}
Equations \eqref{E:Eq-mer-derivative-2} are the derivative formulas of the surface with respect to the parameters $(u,v)$.
Using these formulas we can calculate the Gauss curvature $K$, the curvature of the normal connection $K^{\bot}$, and the mean curvature vector field $H$  of the second kind meridian surface $\M^{II}_m$. 

The  Gauss curvature of $\M^{II}_m$ is expressed as follows: 
\begin{equation*}\label{Eq:meridianSurfGaussCurv2}
K = \frac{\ddot{f}(u)}{f(u)},
\end{equation*}
which means that the Gauss curvature depends only on the parameter $u$ since it is expressed in terms of the function $f(u)$ determining the meridian curve $m$.

Using \eqref{E:Eq-mer-derivative-2} we find that the normal mean curvature vector field $H$  depends on both parameters $u$ and $v$ and is given by the formula:
\begin{equation} \label{E:Eq-mer-H2}
H = \frac{\varkappa (v)}{2f(u)} \, N_1 \mp \frac{f(u) \ddot{f}(u)  + \dot{f}^2(u) -1}{2f(u) \sqrt{\dot{f}^2(u) - 1}} \, N_2.
\end{equation}
\newline

The curvature of the normal connection of $\M^{II}_m$ is:
$$K^{\bot}=0.$$
The above formula implies the following statement.

\begin{prop}
The second kind meridian surface of hyperbolic type $\M^{II}_m$, defined by \eqref{E:Eq-4}, is a surface with flat normal connection.
\end{prop}

\vskip 2mm
\begin{rem}\label{R:Rem1}
Each timelike second kind meridian surface of  hyperbolic type is determined by a meridian curve $m$ and a  spacelike curve $\textbf{c}$ lying on the hyperbolic sphere $\mathbb{H}^2(-1)$ in $\R^3_1$. So, the geometry of the second kind meridian surface and all its invariants are expressed in terms of the curvature $\varkappa_m(u)$ of the meridian curve $m$ and the curvature $\varkappa(v)$ of  $\textbf{c}$.
\end{rem}

\begin{rem}\label{R:Rem2}

(i) If  $\varkappa =0$, then from \eqref{E:Eq-mer-derivative-2} it follows that $\widetilde{\nabla}_X N_1 = \widetilde{\nabla}_X N_2= 0$ and hence, the normal vector field $N_1$ is constant. So, the meridian surface $\M^{II}_m$ lies in the
constant 3-dimensional space  $\R_1^3 = \span \{X,Y,N_2\}$ of $\R_1^4$.

\vskip 1mm
(ii) If $\varkappa_m = 0$, i.e. the meridian curve $m$ is part of a straight  line, then both the Gauss curvature and the normal curvature are zero, which implies that $\M^{II}_m$ is a flat surface with flat normal connection, which is either developable  or  lies in a 3-dimensional space.
\end{rem}

\section{Timelike meridian surfaces of hyperbolic type with constant Gauss curvature}

In this section we describe all timelike first or second kind  meridian surfaces of hyperbolic type in $\R^4_1$ with constant Gauss curvature. First we  shall consider the case when the Gauss curvature is zero, i.e. the surface is flat.

\vskip 1mm
The flat timelike first kind meridian surfaces of hyperbolic type are characterized in the next theorem.

\begin{thm}\label{Th:meridSurfFlat} Let $\M^{I}_m$ be a timelike first kind meridian surface of hyperbolic type, defined by \eqref{E:Eq3}. Then, $\M^{I}_m$ is a flat surface if and only if the meridian curve $m$ is given by:
$$
\begin{array}{l}
\vspace{2mm}
f(u) = a \,u + b; \quad a = \const, \, b =\const;\\
\vspace{2mm}
g(u) = \pm \sqrt{1-a^2} \,u + c; \quad c = \const.
\end{array}
$$
\end{thm}

\begin{proof} 
If $\M^{I}_m$  is defined by \eqref{E:Eq3}, then the Gauss curvature $K$ is expressed by formula \eqref{Eq:meridianSurfGaussCurv}. 
Hence, $K=0$ if and only if $\ddot{f}(u)=0$. The last equality implies that the function $f(u)$ is given by 
\begin{equation} \label{Eq:flat}
f(u) = a \,u + b
\end{equation}
 for some constants $a$ and $b$. 
Since we assume that $\dot{f}(u)^2+\dot{g}(u)^2 = 1$, the function $g(u)$ is determined by $\dot{g}(u) = \pm \sqrt{1-\dot{f}^2(u)}$, and hence, from \eqref{Eq:flat} we obtain 
$$g(u) = \pm \sqrt{1-a^2} \,u + c, $$
where $c = \const$.
\end{proof}

\vskip 1mm
The flat timelike second kind meridian surfaces of hyperbolic type are characterized in the following theorem.

\begin{thm}\label{Th:meridSurfFlat2} Let $\M^{II}_m$ be a timelike meridian surface of hyperbolic type, defined by \eqref{E:Eq-4}. Then,  $\M^{II}_m$ is a flat surface if and only if the meridian curve $m$ is given by:
$$
\begin{array}{l}
\vspace{2mm}
f(u) = a \,u + b; \quad a = \const, \, b =\const;\\
\vspace{2mm}
g(u) = \pm \sqrt{a^2-1} \,u + c; \quad c = \const.
\end{array}
$$
\end{thm}

\begin{proof} 
The proof goes similarly to the proof of Theorem \ref{Th:meridSurfFlat}. In the case of timelike second kind meridian surface of hyperbolic type, we assume that $\dot{f}(u)^2 -\dot{g}(u)^2 = 1$, so the 
 function $g(u)$ is determined by $\dot{g}(u) = \pm \sqrt{\dot{f}^2(u)-1}$ and we get
$$g(u) = \pm \sqrt{a^2-1} \,u + c, $$
where $c = \const$.
\end{proof}

\vskip 1mm
In the next theorem  we describe all timelike first  kind meridian surfaces of hyperbolic type with constant non-zero Gauss curvature.

\begin{thm}\label{Th:meridSurfConstK} Let $\M^{I}_m$  be a timelike meridian surface of hyperbolic type,  defined by \eqref{E:Eq3}. Then, $\M^{I}_m$ has constant non-zero Gauss curvature $K$ if and only if the meridian curve $m$ is given by:
$$
\begin{array}{l}
\vspace{2mm}
f(u) = a_1 \sin \sqrt{K} \,u + a_2 \cos \sqrt{K} \, u,\quad \text{if} \,\, K > 0;\\
\vspace{2mm}
f(u) = a_1 \sinh \sqrt{-K} \,u + a_2 \cosh \sqrt{-K}\, u,\quad \text{if} \,\, K < 0,
\end{array}
$$
where $a_1$ and $a_2$ are constants and the function $g(u)$ is determined by $\dot{g}(u) = \pm \sqrt{1-\dot{f}^2(u)}$.
\end{thm}

\begin{proof}
Let $\M^{I}_m$  be a timelike meridian surface, defined by \eqref{E:Eq3}. Then,  the Gauss curvature $K$ is given by \eqref{Eq:meridianSurfGaussCurv}, which implies that $K$ is constant if and only if the function $f(u)$ satisfies the following differential equation:
$$
\ddot{f}(u) + K f(u) = 0.
$$
The general solution of this equation is given by the formula:
$$
\begin{array}{l}
\vspace{2mm}
f(u) = a_1 \sin \sqrt{K} \,u + a_2 \cos \sqrt{K} \, u,\quad \text{in the case} \,\, K > 0;\\
\vspace{2mm}
f(u) = a_1 \sinh \sqrt{-K} \,u + a_2 \cosh \sqrt{-K}\, u,\quad \text{in the case} \,\, K < 0,
\end{array}
$$
where $a_1$ and $a_2$ are arbitrary constants. Since  $\dot{f}(u)^2+\dot{g}(u)^2 = 1$, the function $g(u)$ is determined by $\dot{g}(u) = \pm \sqrt{1-\dot{f}^2(u)}$.
\end{proof}

\vskip 1mm
In a similar way, it is easy to prove that all timelike second kind meridian surfaces of hyperbolic type with constant non-zero Gauss curvature are described as follows.

\begin{thm}\label{Th:meridSurfConstK2} Let $\M^{II}_m$  be a timelike meridian surface of hyperbolic type,  defined by \eqref{E:Eq-4}. Then, $\M^{II}_m$  has constant non-zero Gauss curvature $K$ if and only if the meridian curve $m$ is given by:
$$
\begin{array}{l}
\vspace{2mm}
f(u) = a_1 \cos \sqrt{-K} \,u + a_2 \sin \sqrt{-K} \, u,\quad \text{if} \,\, K < 0;\\
\vspace{2mm}
f(u) = a_1 \cosh \sqrt{K} \,u + a_2 \sinh \sqrt{K}\, u,\quad \text{if} \,\, K > 0,
\end{array}
$$
where $a_1$ and $a_2$ are constants and the function $g(u)$ is determined by $\dot{g}(u) = \pm \sqrt{\dot{f}^2(u)-1}$.
\end{thm}

\vskip 3mm

\section{Minimal timelike meridian surfaces of hyperbolic type}

In this section we classify all minimal timelike first or second kind meridian surfaces of hyperbolic type.

\begin{thm}\label{Th:meridSurfMin-1} 
Let $\M^{I}_m$ be a timelike meridian surface of hyperbolic type,  defined by \eqref{E:Eq3}. Then, $\M^{I} _m$ is minimal if and only if the curve 
$\textbf{c}$ on $\mathbb{S}^2_1(1)$ has zero spherical curvature and the meridian curve $m$ is given by:
$$\begin{array}{l}
\vspace{2mm}
f(u) = \pm \sqrt{u^2 +2au+b};\\
\vspace{2mm}
g(u) = \pm \sqrt{b-a^2} \ln |u+a + \sqrt{u^2+ 2au + b}| + c,
\end{array}
$$
where $a = \const$, $b=\const$, $c = \const$, $b-a^2>0$. 
\end{thm}

\begin{proof} 
Using that the mean curvature vector field $H$ of $\M^{I}_m$  is given by formula \eqref{E:Eq-mer-H}, we get that the first kind meridian surface of hyperbolic type is minimal if and only if the curvature of $\textbf{c}$ is $\varkappa = 0$ and the function $f(u)$ satisfies the following differential equation 
$$f \ddot{f} + \dot{f}^2 - 1= 0.$$
The general solution of this differential equation is given by the  formula:
$$
f(u) = \pm \sqrt{u^2+2a u + b}, \quad a= \const, \, b = \const.
$$
Using that $\dot{g} = \pm\sqrt{1-\dot{f}^2}$, we get  the following equation for $g(u)$:
$$
\dot{g}= \pm \frac{\sqrt{b -a^2}}{\sqrt{u^2+2a u + b}},
$$
and after integrating we obtain
$$g(u) = \pm \sqrt{b-a^2} \ln |u+a + \sqrt{u^2+ 2au + b}| + c,  \quad c = const.$$

\end{proof}

\vskip 3mm

\begin{thm}\label{Th:meridSurfMin-2} 
Let $\M^{II}_m$ be a timelike meridian surface of hyperbolic type,  defined by \eqref{E:Eq-4}. Then, $\M^{II}_m$ is minimal if and only if the curve 
$\textbf{c}$ on $\mathbb{H}^2(-1)$ has zero curvature and the meridian curve $m$ is given by:
$$\begin{array}{l}
\vspace{2mm}
f(u) = \pm \sqrt{u^2 +2au+b};\\
\vspace{2mm}
g(u) = \pm \sqrt{a^2-b} \ln |u+a + \sqrt{u^2+ 2au + b}| + c,
\end{array}
$$
where $a = \const$, $b=\const$, $c = \const$, $a^2-b>0$. 
\end{thm}

\vskip 3mm
The proof of Theorem \ref{Th:meridSurfMin-2}  goes similarly to the proof of Theorem \ref{Th:meridSurfMin-1}. 

\vskip 3mm

\begin{cor}
There are no minimal timelike meridian surfaces of hyperbolic type in  $\R^4_1$ other than surfaces lying in a hyperplane of $\R^4_1$.
\end{cor}

\begin{proof}
If $\M^{I}_m$  (resp. $\M^{II}_m$) is minimal, then the  curvature of $\textbf{c}$ is $\varkappa = 0$. Hence, according to Remark \ref{R:Rem1} (resp.  Remark \ref{R:Rem2}), the meridian surface $\M^{I}_m$  (resp. $\M^{II}_m$) lies in the constant 3-dimensional space  $\R_1^3 = \span \{X,Y,N_2\}$ of $\R_1^4$, i.e. the surface lies in a hyperplane of $\R^4_1$.

\end{proof}

For recent results on minimal timelike surfaces in the Minkowski 3-space one can refer to \cite{Ch-D-M}, \cite{L-M-M}, \cite{C-O}, \cite{Pa-Py}.

\vskip 3mm

\section{Timelike meridian surfaces of hyperbolic type with constant mean curvature}

In the next theorem we classify all timelike first kind meridian surfaces of hyperbolic type  with non-zero constant mean curvature (CMC-surfaces).

\begin{thm}\label{Th:meridSurfConstH} 
Let $\M^{I}_m$  be a timelike meridian surface of hyperbolic type,  defined by \eqref{E:Eq3}. Then, $\M^{I}_m$  has constant mean curvature vector field, i.e. $|| H || = a = \const, a \neq 0$,  if and only if the curve $\textbf{c}$ on  $\mathbb{S}^2_1(1)$ has constant non-zero spherical curvature $\varkappa = \const = b, \, b\neq 0$ and the meridian curve $m$ is given by $\dot{f} = \varphi (f)$, where
$$
\varphi(t) = \pm  \sqrt{1- \frac{1}{t^2} \left(c \mp \frac{t}{2}  \sqrt{4 a^2 t^2-b^2} \pm \frac{b^2}{4a} \ln | 2a t +   \sqrt{4 a^2 t^2-b^2}| \right)^2}, \quad c = \const,
$$
and the function $g(u)$ is defined by $\dot{g} = \pm\sqrt{1-\dot{f}^2}$.
\end{thm}

\begin{proof} 
Using that the mean curvature vector field $H$ of the meridian surface of hyperbolic type,  defined by \eqref{E:Eq3}, is given by formula  \eqref{E:Eq-mer-H}, we conclude that  $|| H || = \const = a$ if and only if the following equality holds:
$$
\varkappa^2 = \frac{4a^2 f^2  (1-\dot{f}^2 ) - (f \ddot{f}-1 + \dot{f}^2)^2}{1-\dot{f}^2}.
$$
The left-hand side of the above equality is a function of the parameter $v$, while the right-hand side is a function of $u$, which implies that 
\begin{equation}\label{Eq:meridianSurfMeanCurvConstant}
\begin{array}{l}
\vspace{2mm}
\varkappa = \const = b, \,\, b \neq 0; \\
\vspace{2mm}
4a^2 f^2  (1-\dot{f}^2 ) - (f \ddot{f}-1 + \dot{f}^2)^2 = b^2 (1-\dot{f}^2).
\end{array}
\end{equation}
The first equality of \eqref{Eq:meridianSurfMeanCurvConstant} means that the  curve $\textbf{c}$ has constant spherical curvature $\varkappa = b$. The second equality of \eqref{Eq:meridianSurfMeanCurvConstant} implies:
\begin{equation}\label{Eq:meridianSurfMeanCurvConstantDifEq}
(f \ddot{f}-1 + \dot{f}^2)^2 = (1-\dot{f}^2)(4a^2 f^2 - b^2).
\end{equation}
The solutions to the above differential equation can be found by setting
$\dot{f} = \varphi(f)$ in equation  \eqref{Eq:meridianSurfMeanCurvConstantDifEq}. Then we obtain that the function $\varphi = \varphi(t)$ is a solution to the following differential equation:
$$
1 + \varphi^2 + \frac{t}{2} (\varphi^2)' = \pm \sqrt{\varphi^2 + 1}\sqrt{4a^2 t^2 - b^2},
$$
whose general solution is given by the formula:
\begin{equation}\label{Eq:meridianSurfMeanCurvSol}
\varphi(t) = \pm  \sqrt{1- \frac{1}{t^2} \left(c \mp \frac{t}{2}  \sqrt{4 a^2 t^2-b^2} \pm \frac{b^2}{4a} \ln | 2a t +   \sqrt{4 a^2 t^2-b^2}| \right)^2},
\end{equation}
where $c = \const$. The function $f$ is determined by $\dot{f} = \varphi(f)$ and formula \eqref{Eq:meridianSurfMeanCurvSol}, and the function $g$ is defined by $\dot{g} = \pm \sqrt{1-\dot{f}^2}$.

\end{proof}

\vskip 3mm
Similarly to the proof of Theorem \ref{Th:meridSurfConstH},  we obtain the classification of all timelike second kind meridian surfaces of hyperbolic type  with non-zero constant mean curvature.

\begin{thm}\label{Th:meridSurfConstH2} 
Let $\M^{II}_m$  be a  timelike meridian surface of hyperbolic type,  defined by \eqref{E:Eq-4}. Then, $\M^{II}_m$   has constant mean curvature vector field, i.e. $|| H || = a = \const, a \neq 0$,  if and only if the curve $\textbf{c}$ on $\mathbb{H}^2(-1)$ has constant non-zero curvature $\varkappa = \const = b, \, b\neq 0$ and the meridian curve $m$ is given by $\dot{f} = \varphi (f)$, where
$$
\varphi(t) = \pm  \sqrt{1+ \frac{1}{t^2} \left(c \pm \frac{t}{2}  \sqrt{4 a^2 t^2-b^2} \mp \frac{b^2}{4a} \ln | 2a t +   \sqrt{4 a^2 t^2-b^2}| \right)^2}, \quad c = \const,
$$
and the function $g(u)$ is defined by $\dot{g} = \pm\sqrt{\dot{f}^2-1}$.
\end{thm}

\vskip 3mm

\section{Timelike meridian surfaces of hyperbolic type with parallel mean curvature vector field}

In this section, we  consider non-minimal surfaces,
i.e. we assume that $H \neq 0$.
The mean curvature vector field $H$ of a  timelike first kind meridian surface of hyperbolic type  is expressed by formula  \eqref{E:Eq-mer-H}. 
Using the derivative formulas \eqref{E:Eq-mer-derivative-1} we calculate $D_X H$ and $D_Y H$ and obtain the following equalities:
\begin{equation}\label{Eq:MeanCurvDeriv}
\begin{array}{l}
\vspace{.2cm}
D_X H = \frac{\varkappa \dot{f}}{2f^2} \,N_1 \pm \frac{\partial}{\partial u} \left(\frac{f \ddot{f} + \dot{f}^2-1}{2f \sqrt{1-\dot{f}^2}}\right)\,N_2; \\
\vspace{.2cm} 
D_Y H = -\frac{\varkappa'}{2f^2} \,N_1.
\end{array}
\end{equation}
Recall that the mean curvature vector field  $H$ is parallel in the normal bundle if and only if $D_X H = D_Y H = 0$.

\vskip 2mm
In the next theorem we describe all timelike first kind meridian surfaces of hyperbolic type with parallel mean curvature vector field.

\begin{thm}\label{Th:meridSurfParralelH}
Let $\M^{I}_m$  be a timelike meridian surface of hyperbolic type,  defined by \eqref{E:Eq3}. Then,  $\M^{I}_m$  has parallel mean curvature vector field, if and only if one of the following cases holds:

\hskip 10mm (i)  $\textbf{c}$ has zero spherical curvature and the meridian curve 
$m$ is determined by $\dot{f} = \varphi(f)$ where 
\begin{equation}
\varphi(t) = \pm \frac{1}{t} \sqrt{t^2-(c - a\, t^2)^2}, \quad a = const
\neq 0, \quad c = const,  \notag
\end{equation}
$g(u)$ is defined by $\dot{g} = \pm \sqrt{1-\dot{f}^2}$. In this case, $\M^{I}_m$  is a non-flat CMC-surface lying in a hyperplane of $\mathbb{E}^4_1$.

\hskip 10mm (ii)  $\textbf{c}$ has non-zero constant spherical curvature and the
meridian curve $m$ is determined by $f(u) = a$, $g(u) = \pm u + b$, where $a = const
\neq 0$, $b=const$. In this case, $\M^{I}_m$  is a flat CMC-surface lying in
a hyperplane of $\mathbb{R}^4_1$.
\end{thm}

\begin{proof}

Using  formulas \eqref{Eq:MeanCurvDeriv}, we get that the surface has parallel mean curvature
vector field if and only if the following conditions are fulfilled: 
\begin{equation}  \label{E:Eq-7}
\begin{array}{l}
\vspace{2mm} \varkappa'(v) = 0; \\ 
\vspace{2mm} \varkappa \dot{f} = 0; \\ 
\vspace{2mm} \frac{{f \ddot{f}}+ \dot{f}^2 - 1}{2f \sqrt{1-\dot{f}^2}} =const.%
\end{array}%
\end{equation}
The first equality of \eqref{E:Eq-7} implies that  $\varkappa = \const$.  The second equality of \eqref{E:Eq-7} shows that
there are two possible cases:

\vskip1mm 
Case (i): The curvature of $\textbf{c}$ is $\varkappa =0$. In this case, 
$\widetilde{\nabla}_X N_1 = 0, \; \widetilde{\nabla}_Y N_1 = 0$, which  implies that  the surface $\M^{I}_m$  lies in the constant 3-dimensional space $\R^3_1 = \span \{X,Y,N_2\}$. 
From the third equality of \eqref{E:Eq-7} we get that the function $f(u)$  is a solution to the following differential equation: 
$$f \ddot{f}+ \dot{f}^2 - 1 =  2af \sqrt{1-\dot{f}^2 },$$ 
where $a=const$, $a\neq 0$. 
Setting $\dot{f}=\varphi (f)$, we obtain that the function $\varphi =\varphi (t)$ is a solution to the equation: 
\begin{equation} \label{E:Eq-9}
\frac{t}{2}\,(\varphi ^{2})^{\prime }+ \varphi^{2} - 1= 2at\sqrt{1-\varphi ^{2}}.
\end{equation}%
Now, by denoting $z(t)=\sqrt{1-\varphi ^{2}(t)}$, the last equation is transformed into the equation:
\begin{equation*}
z^{\prime }(t)+\frac{1}{t}\,z(t)= -2a,
\end{equation*}%
whose general solution is given by the formula $z(t)=%
\frac{c- at^{2}}{t}$, $c=const$. Finally,  we obtain that the general solution of  \eqref{E:Eq-9} is 
\begin{equation*}
\varphi (t)=\pm \frac{1}{t}\sqrt{t^{2}-(c- a\,t^{2})^{2}}.  
\end{equation*}%
In this case, $\M^{I}_m$ is a CMC-surface, since the mean curvature vector field is expressed as $H = \pm a\, N_2$, which implies  $\langle H,H\rangle =a^{2}=const$. Calculating  the Gauss curvature $K$ we obtain  $K =\mp \frac{c^2-a^2 t^4}{t^4}$. Consequently, 
 $\M^{I}_m$  is a non-flat CMC-surface lying in a constant hyperplane of $\R^4_1$.

\vskip 2mm Case (ii):  The curvature of $\textbf{c}$ is non-zero, i.e. $\varkappa \neq 0$. Then, from the second equality of \eqref{E:Eq-7} we obtain that the function $f(u)$ has the following form: 
$$f(u) = a, \quad a = const \neq 0.$$
Now, having in mind that $\dot{f}^2+\dot{g}^2 = 1$, we get $g(u) = \pm u+b$, $b = const$. In this case, the mean curvature vector field is expressed as:
\begin{equation*}
H = -\frac{\varkappa}{2a}\, N_1 \mp \frac{1}{2a} \, N_2,
\end{equation*}
which implies that $\langle H, H \rangle = \frac{1+\varkappa^2}{4a^2}=const$. Hence,  $\M^{I}_m$   is a CMC-surface. 
Moreover, for the Gauss curvature we obtain  $K=0$, i.e.   $\M^{I}_m$   is a flat surface. We will prove that the surface lies in a constant 3-dimensional space. Indeed, let us consider the normal vector fields 
$$\bar{N}_1 = \frac{1}{\sqrt{\varkappa^2+1}} (N_1 \pm \varkappa\,N_2); \quad \bar{N}_2 = \frac{1}{\sqrt{\varkappa^2+1}} (\mp \varkappa\,N_1 + N_2).$$ 
By use of equalities  \eqref{E:Eq-mer-derivative-1} we obtain $\widetilde{\nabla}_X \bar{N}_1 = \widetilde{\nabla}_Y \bar{N}_1 = 0$, which implies that  $\M^{I}_m$  lies in the constant 3-dimensional space $\R^3_1 = \span \{X,Y,\bar{N}_2\}$. 
Consequently, in this case the meridian surface  $\M^{I}_m$  is a flat CMC-surface lying in a hyperplane of $\mathbb{E}^4_1$.

\vskip 2mm Conversely, if one of the cases (i) or (ii) in the theorem
holds true, then by direct computations we get that $D_X H = D_YH = 0$, which means that
the surface has parallel mean curvature vector field.

\end{proof}

\vskip 3mm

In a  similar way, we can describe all timelike second kind meridian surfaces of hyperbolic type with parallel mean curvature vector field.

Let $\M^{II}_m$  be a timelike second type meridian surface of hyperbolic type. Then, its mean curvature vector field $H$ is expressed by  \eqref{E:Eq-mer-H2}. Using \eqref{E:Eq-mer-derivative-2} we obtain the following expressions for $D_X H$ and $D_Y H$:
\begin{equation*}\label{Eq:MeanCurvDeriv2}
\begin{array}{l}
\vspace{.2cm}
D_X H = -\frac{\varkappa \dot{f}}{2f^2} \,N_1 \mp \frac{\partial}{\partial u} \left(\frac{f \ddot{f} + \dot{f}^2-1}{2f \sqrt{\dot{f}^2-1}}\right)\,N_2; \\
\vspace{.2cm} 
D_Y H = \frac{\varkappa'}{2f^2} \,N_1.
\end{array}
\end{equation*}

\vskip 2mm
The timelike  second type meridian surfaces of hyperbolic type with parallel mean curvature vector field are classified in the next theorem.

\begin{thm}\label{Th:meridSurfParralelH}
Let $\M^{II}_m$  be a timelike meridian surface of hyperbolic type,  defined by \eqref{E:Eq-4}. Then,  $\M^{II}_m$   has parallel mean curvature vector field, if and only if one of the following cases holds:

\hskip 10mm (i)  $\textbf{c}$ has zero curvature and the meridian  curve
$m$ is determined by $\dot{f} = \varphi(f)$, where 
\begin{equation}
\varphi(t) = \pm \frac{1}{t} \sqrt{t^2+(C + a\, t^2)^2}, \quad a = const
\neq 0, \quad C = const,  \notag
\end{equation}
$g(u)$ is defined by $\dot{g} = \pm \sqrt{\dot{f}^2-1}$. In this case, $\M^{II}_m$   is a non-flat CMC-surface lying in a hyperplane of $\mathbb{E}^4_1$.

\hskip 10mm (ii)  $\textbf{c}$ has non-zero constant curvature and the
meridian curve $m$ is determined by $f(u) = a$, $g(u) = \pm u + b$, where $a = const
\neq 0$, $b=const$. In this case, $\M^{II}_m$   is a flat CMC-surface lying in
a hyperplane of $\mathbb{E}^4_1$.
\end{thm}

\vskip 3mm
\begin{cor}
There are no timelike meridian surfaces of hyperbolic type with parallel mean curvature vector field other than CMC-surfaces lying in a hyperplane of $\R^4_1$.
\end{cor}

\section{Timelike meridian surfaces of hyperbolic type with parallel normalized mean curvature vector field}

This section is devoted to the class of  timelike meridian surfaces of hyperbolic type  with parallel normalized mean curvature vector field, but non-parallel $H$. 

Let $\M^{I}_m$  be a timelike  meridian surface of hyperbolic type, defined by \eqref{E:Eq3}, and assume that $\langle H, H \rangle \neq 0$ and  $H$ is not parallel. Then, the normalized mean curvature vector field of $\M^{I}_m$ is defined as $H_0 =\frac{H}{\sqrt{\langle H, H \rangle}}$. Having in mind   \eqref{E:Eq-mer-H}, we obtain:
\begin{equation}  \label{E:Eq-H0}
H_0 = -\frac{\varkappa \sqrt{1-\dot{f}^2}}{\sqrt{\varkappa^2 (1-\dot{f}^2) + (f \ddot{f}+ \dot{f}^2 - 1)^2}} \,N_1 \pm \frac{f \ddot{f}+ \dot{f}^2 - 1}{\sqrt{\varkappa^2 (1-\dot{f}^2) + (f \ddot{f}+ \dot{f}^2 - 1)^2}} \,N_2.
\end{equation}
Using \eqref{E:Eq-mer-derivative-1} and \eqref{E:Eq-H0} we get the following expressions for the vector fields $D_X H_0$ and $D_Y H_0$: 
\begin{equation*}\label{Eq:NormalizedMeanCurvDeriv}
\begin{array}{l}
\vspace{.2cm}
D_X H_0 \!=\! -\frac{\partial}{\partial u} \!\left(\!\! \frac{\varkappa \sqrt{1-\dot{f}^2}}{\sqrt{\varkappa^2 (1-\dot{f}^2) + (f \ddot{f}+ \dot{f}^2 - 1)^2}}  \!\! \right) \!\! N_1 \pm \frac{\partial}{\partial u} \!\left(\! \!\frac{f \ddot{f}+ \dot{f}^2 - 1}{\sqrt{\varkappa^2 (1-\dot{f}^2) + (f \ddot{f}+ \dot{f}^2 - 1)^2}} \!\!\right) \!\! N_2; \\
\vspace{.2cm} 
D_Y H_0 \!= \!-\frac{1}{f} \frac{\partial}{\partial v} \!\left(\!\! \frac{\varkappa \sqrt{1-\dot{f}^2}}{\sqrt{\varkappa^2 (1-\dot{f}^2) + (f \ddot{f}+ \dot{f}^2 - 1)^2}}  \!\! \right)\! \! N_1 \pm \frac{1}{f}\frac{\partial}{\partial v} \!\left(\! \!\frac{f \ddot{f}+ \dot{f}^2 - 1}{\sqrt{\varkappa^2 (1-\dot{f}^2) + (f \ddot{f}+ \dot{f}^2 - 1)^2}}\! \!\right) \!\! N_2.
\end{array}
\end{equation*}

The last  formulas imply that  $\M^{I}_m$  has  parallel normalized mean curvature vector field ($D_X H_0 = D_Y H_0 = 0$) if and only if the following two equalities hold true:
\begin{equation}\label{Eq:NormalizedMeanCurv-cond}
\begin{array}{l}
\vspace{.2cm}
\frac{\varkappa \sqrt{1-\dot{f}^2}}{\sqrt{\varkappa^2 (1-\dot{f}^2) + (f \ddot{f}+ \dot{f}^2 - 1)^2}} = \const = \alpha; \\
\vspace{.2cm} 
\frac{f \ddot{f}+ \dot{f}^2 - 1}{\sqrt{\varkappa^2 (1-\dot{f}^2) + (f \ddot{f}+ \dot{f}^2 - 1)^2}}  =\const =\beta,
\end{array}
\end{equation}
for some constants $\alpha$ and $\beta$.

\vskip 2mm

\begin{thm}\label{Th:meridSurfParallelNorm}
Let  $\M^{I}_m$ be a timelike  meridian surface of hyperbolic type, defined by \eqref{E:Eq3}. Then,   $\M^{I}_m$ has parallel normalized mean curvature vector field, but non-parallel mean curvature vector, if and only if one of the following cases holds:

\hskip 10mm (i)  $\varkappa \neq 0$ and the  meridian curve  $m$ is defined by
$$\begin{array}{l}
\vspace{2mm}
f(u) = \pm \sqrt{u^2 +2au+b}, \\
\vspace{2mm}
g(u) = \pm \sqrt{b-a^2} \, \ln |(u+a)+\sqrt{u^2+2au+b}| + c,
\end{array}$$ 
where $a = \const$, $b=\const$, $c = \const$, $b-a^2 >0$.

\hskip 10mm (ii)  the curve $\textbf{c}$   has non-zero constant spherical curvature  and the meridian curve $m$ is determined by 
\begin{equation} \notag
\begin{array}{l}
\vspace{2mm}
g(t) =  \frac{at^2}{2} - ct + b, \,\,  \\
\vspace{2mm}
f(t) = \pm \left [  \frac{\sqrt{1-(at - c)^2} (2at-c)}{4a} + \frac{1}{2a} \arcsin (at-c) \right ],
\end{array}
\end{equation}
where $a = \const$, $b=\const$, $c = \const \neq 0,\; c^2 \neq \varkappa^2$.
\end{thm}

\begin{proof}
Let  $\M^{I}_m$ be a timelike first kind meridian surface of hyperbolic type with parallel normalized mean curvature vector field, i.e. $D_X H_0 = D_Y H_0 =0$, or equivalently, equalities  \eqref{Eq:NormalizedMeanCurv-cond} hold true. We will consider the following two possible cases.

\vskip 1mm
Case (i):  $f \ddot{f} + \dot{f}^2 - 1= 0$. In this case, from \eqref{E:Eq-H0} it follows that the normalized mean curvature vector field is $H_0 = N_1$ and the mean curvature vector field is 
$H = -\frac{\varkappa}{2f} \,N_1$. Hence, $\varkappa \neq 0$ (otherwise the surface is minimal). The general solution to the differential equation 
$f \ddot{f} + \dot{f}^2 - 1= 0$ is given by
$$f(u) = \pm \sqrt{u^2 +2au+b},$$
 where $a=\const$, $b =\const$. Using that  $\dot{g} = \pm \sqrt{1-\dot{f}^2}$, after integration we obtain that the function  $g(u)$ is expressed as follows:
$$g(u) = \pm \sqrt{b-a^2} \, \ln |(u+a)+\sqrt{u^2+2au+b}| + c,$$
where  $c = \const$, $b-a^2 >0$.

\vskip 1mm
Case (ii):  $f \ddot{f} + \dot{f}^2 - 1 \neq 0$ in a sub-interval $I_0 \subset I \subset \R$. In this case, formulas \eqref{Eq:NormalizedMeanCurv-cond} imply
\begin{equation} \label{E:Eq-13}
\frac{\beta}{\alpha} \,\varkappa =  \frac{f \ddot{f} + \dot{f}^2 - 1}{\sqrt{1-\dot{f}^2}},  \quad \alpha \neq 0, \; \beta \neq 0.
\end{equation}
Since the left-hand side of \eqref{E:Eq-13} is a function of $v$, while the right-hand side is a function of $u$, it follows that 
\begin{equation*} \notag
\begin{array}{l}
\vspace{2mm}
\frac{f \ddot{f} + \dot{f}^2 - 1}{\sqrt{1-\dot{f}^2}} = c, \quad c = \const \neq 0;\\
\varkappa = \frac{\alpha}{\beta}\, c = \const.
\end{array}
\end{equation*}
In this case, for the mean curvature vector field we have $\langle H,H \rangle  = \frac{\varkappa^2 + c^2}{4f^2} \neq 0$. 
To find the solutions of the differential
equation:
\begin{equation} \label{E:Eq-14}
f \ddot{f} + \dot{f}^2 - 1 = c \sqrt{1-\dot{f}^2},
\end{equation}
we set $\dot{f} = \varphi (f)$ and obtain
that the function $\varphi  = \varphi (t)$ satisfies
\begin{equation} \label{E:Eq-15}
\frac{t}{2} \,(\varphi ^2)' + \varphi ^2 - 1 = c \sqrt{1-\varphi ^2}.
\end{equation}
By setting $z(t) = \sqrt{1-\varphi ^2(t)}$, the last equation is transformed to 
\begin{equation} \notag
z' - \frac{1}{t}\, z = \frac{c}{t},
\end{equation}
whose solution  is given by the formula $z(t) = at - c$, $a = \const$.
Consequently, the general solution of \eqref{E:Eq-15} is given by the formula
\begin{equation} \notag
\varphi(t) = \pm \sqrt{1-(a t-c)^2}.
\end{equation}
Using that $\dot{f} = \varphi$ and $\dot{g} = \sqrt{1-\dot{f}^2}$, from the above formula we obtain:
$$
 g =  \frac{at^2}{2} - ct + b, \,\, b = \const,
$$
$$
f = \pm \int{\sqrt{1-(at - c)^2}} \, \rm{d}t = \pm \left [  \frac{\sqrt{1-(at - c)^2} (2at-c)}{4a} + \frac{1}{2a} \arcsin (at-c) \right ].
$$

\vskip 1mm
Conversely, if one of the cases (i) or (ii)  in the theorem  holds true, then by direct computations we obtain  that
$D_X H_0 = D_Y H_0 = 0$, which means that the surface has parallel normalized  mean curvature vector field.
Moreover, in case (i) we get
$$D_XH = \frac{\varkappa \dot{f}}{2f^2} \,N_1; \qquad D_YH = -\frac{\varkappa'}{2f^2} \,N_1,$$
and in case (ii) we get
$$D_XH = \frac{\varkappa \dot{f}}{2f^2} \,N_1 \pm \frac{C\dot{f}}{2f^2} \,N_2; \qquad D_YH = 0.$$
Hence, in both cases the mean curvature vector field $H$ is not parallel in the normal bundle, since $\varkappa \neq 0$ and $\dot{f} \neq 0$. 

\end{proof}

\vskip 3mm

In a  similar way, we can describe all timelike second kind meridian surfaces of hyperbolic type with parallel normalized  mean curvature vector field.

Let  $\M^{II}_m$  be a timelike  meridian surface of hyperbolic type, defined by \eqref{E:Eq-4}, and the mean curvature vector field $H$ is non-parallel, $\langle H, H \rangle \neq 0$. The normalized mean curvature vector field is $H_0 =\frac{H}{\sqrt{\langle H, H \rangle}}$. Using  \eqref{E:Eq-mer-H2} we get:
\begin{equation}  \label{E:Eq-H0-2}
H_0 = \frac{\varkappa \sqrt{\dot{f}^2-1}}{\sqrt{\varkappa^2 (\dot{f}^2-1) + (f \ddot{f}+ \dot{f}^2 - 1)^2}} \,N_1 \mp \frac{f \ddot{f}+ \dot{f}^2 - 1}{\sqrt{\varkappa^2 (\dot{f}^2-1) + (f \ddot{f}+ \dot{f}^2 - 1)^2}} \,N_2.
\end{equation}

The surface  $\M^{II}_m$  has  parallel normalized mean curvature vector field if and only if 
 $D_X H_0 = D_Y H_0 = 0$.  
Using \eqref{E:Eq-mer-derivative-2} and \eqref{E:Eq-H0-2} we obtain:
\begin{equation*}\label{Eq:NormalizedMeanCurvDeriv}
\begin{array}{l}
\vspace{.2cm}
D_X H_0 = \frac{\partial}{\partial u} \!\left(\!\! \frac{\varkappa \sqrt{\dot{f}^2-1}}{\sqrt{\varkappa^2 (\dot{f}^2-1) + (f \ddot{f}+ \dot{f}^2 - 1)^2}}  \!\! \right) \! N_1 \mp \frac{\partial}{\partial u} \!\left(\! \!\frac{f \ddot{f}+ \dot{f}^2 - 1}{\sqrt{\varkappa^2 (\dot{f}^2-1) + (f \ddot{f}+ \dot{f}^2 - 1)^2}} \!\!\right) \! N_2; \\
\vspace{.2cm} 
D_Y H_0 = \frac{1}{f} \frac{\partial}{\partial v} \!\left(\!\! \frac{\varkappa \sqrt{\dot{f}^2-1}}{\sqrt{\varkappa^2 (\dot{f}^2-1) + (f \ddot{f}+ \dot{f}^2 - 1)^2}}  \!\! \right) \! N_1 \mp \frac{1}{f}\frac{\partial}{\partial v} \!\left(\! \!\frac{f \ddot{f}+ \dot{f}^2 - 1}{\sqrt{\varkappa^2 (\dot{f}^2-1) + (f \ddot{f}+ \dot{f}^2 - 1)^2}}\! \!\right) \! N_2.
\end{array}
\end{equation*}

Hence,   $\M^{II}_m$ has  parallel normalized mean curvature vector field if and only if the following equalities hold true:
\begin{equation*}\label{Eq:NormalizedMeanCurv-cond2}
\begin{array}{l}
\vspace{.2cm}
\frac{\varkappa \sqrt{\dot{f}^2-1}}{\sqrt{\varkappa^2 (\dot{f}^2-1) + (f \ddot{f}+ \dot{f}^2 - 1)^2}} = \const = \alpha; \\
\vspace{.2cm} 
\frac{f \ddot{f}+ \dot{f}^2 - 1}{\sqrt{\varkappa^2 (\dot{f}^2-1) + (f \ddot{f}+ \dot{f}^2 - 1)^2}}  =\const =\beta,
\end{array}
\end{equation*}
for some constants $\alpha$ and $\beta$.

\vskip 2mm
The timelike second kind meridian surfaces of hyperbolic type  with parallel normalized mean curvature vector field are described in the next theorem.

\begin{thm}\label{Th:meridSurfParallelNorm-2}
Let $\M^{II}_m$ be a timelike  meridian surface of hyperbolic type, defined by \eqref{E:Eq-4}. Then,  $\M^{II}_m$ has parallel normalized mean curvature vector field, but non-parallel mean curvature vector, if and only if one of the following cases holds:

\hskip 10mm (i)  $\varkappa \neq 0$ and the  meridian curve $m$ is defined by
$$\begin{array}{l}
\vspace{2mm}
f(u) = \pm \sqrt{u^2 +2au+b}, \\
\vspace{2mm}
g(u) = \pm \sqrt{a^2-b} \, \ln |(u+a)+\sqrt{u^2 +2au+b}| + c,
\end{array}$$ 
where $a = \const$, $b=\const$, $c = \const$, $a^2-b>0$.

\hskip 10mm (ii)  the curve $\textbf{c}$   has non-zero constant  curvature  and the meridian curve $m$ is determined by 
\begin{equation} \notag
\begin{array}{l}
\vspace{2mm}
g(t) =  a \ln t + ct + b, \\
\vspace{2mm}
\dot{f}(t) = \pm\sqrt{\left (\frac{a}{t}+c\right)^2+ 1},
\end{array}
\end{equation}
where $a = \const$, $b=\const$, $c = \const \neq 0,\; c^2 \neq \varkappa^2$.
\end{thm}

The proof of Theorem \ref{Th:meridSurfParallelNorm-2} is similar to the proof of Theorem \ref{Th:meridSurfParallelNorm}.
\vskip 2mm

The results in Theorem \ref{Th:meridSurfParallelNorm} and Theorem \ref{Th:meridSurfParallelNorm-2} can be used to find explicit solutions to the background systems of natural partial differential equations describing the timelike surfaces with parallel normalized
mean curvature vector field in $\R^4_1$ \cite{BM-TJM}.

\vskip 6mm 
\textbf{Acknowledgments:}
The  authors are partially supported by the National Science Fund, Ministry of Education and Science of Bulgaria under contract KP-06-N82/6.

\end{document}